\documentclass[12pt,reqno]{amsart}
\usepackage{amssymb}
\usepackage[utf8]{inputenc}
\usepackage[pagewise]{lineno}
\usepackage[letterpaper,left=1.25in,right=1.25in,top=1.1in,bottom=0.93in]{geometry}
\usepackage{amsthm}
\usepackage{tikz}
\usepackage{mathtools}

\usepackage{wrapfig}
\usepackage{caption}
\usepackage{float}
\usepackage{units}
\usepackage{todonotes}
\usepackage{comment}
\usepackage{hyperref}
\usepackage{float}

\usepackage[msc-links, initials]{amsrefs}
\usepackage{url}

\makeatletter
\newcommand{\addresseshere}{%
  \enddoc@text\let\enddoc@text\relax
}
\makeatother

\DeclareMathOperator{\lcm}{lcm}

\newtheorem{theorem}{Theorem}[section]

\newtheorem{question}[theorem]{Question}

\newtheorem{lemma}[theorem]{Lemma}
\newtheorem{prop}[theorem]{Proposition}

\title{Permutations and the divisor graph of $[1,n]$}
\author{Nathan McNew }
\address{Department of Mathematics, Towson University, Towson, MD 21252, USA}
\email{nmcnew@towson.edu}

\begin{document}

\begin{abstract}
    Let $S_{\rm div}(n)$ denote the set of permutations $\pi$ of $n$ such that for each $1\leq j \leq n$ either $j \mid \pi(j)$ or $\pi(j) \mid j$.  These permutations can also be viewed as vertex-disjoint directed cycle covers of the divisor graph $\mathcal{D}_{[1,n]}$ on vertices $v_1, \ldots, v_n$ with an edge between $v_i$ and $v_j$ if $i\mid j$ or $j \mid i$.  We improve on recent results of Pomerance by showing $c_d = \lim_{n \to \infty }\left(\# S_{\rm div}(n)\right)^{1/n}$ exists and that $2.069<c_d<2.694$.  We also obtain similar results for the set $S_{\rm lcm}(n)$ of permutations where $\lcm(j,\pi(j))\leq n$ for all $j$.  The results rely on a graph theoretic result bounding the number of vertex-disjoint directed cycle covers, which may be of independent interest.
\end{abstract}
\maketitle

\section{Introduction}
Several recent papers have investigated permutations $\pi$ of the integers $[1,n]=\{1,2,\ldots n\}$ where the value $\pi(i)$ is related to the index $i$ in an arithmetically interesting way.  

In \cite{pom1} Pomerance introduces the \textit{coprime permutations} (permutations satisfying the condition $\gcd(i,\pi(i))=1$ for all $1\leq i \leq n$). He shows, for large $n$, that the count of such permutations is bounded between $n!/c_1^n$ and $n!/c_2^n$ with $c_1 = 3.73$ and $c_2 = 2.5$. That paper also includes a conjecture made by the author of this paper that the count of co-prime permutations of $n$, as $n \to \infty$, is $n!/c_o^{n+o(n)}$ where $c_o = 
\prod_{p}
\frac{p(p - 2)^{1-2/p}}{(p - 1)^{2-2/p}} =  2.65044\ldots.$
This conjecture was proven a few days later in a paper by Sah and Sawhney \cite{SS}. 

Coprime permutations are a special case of coprime mappings, bijective maps $f:~A\to B$ between two sets of integers with the property that $\gcd(a,f(a))=1$ for every $a \in A$.  Coprime mappings have been considered by many authors.  Newman conjectured that there always exists a coprime mapping from $[1,n]$ to any set of $n$ consecutive integers.  This was proven by Pomerance and Selfridge \cite{CS80}. The special case of coprime mappings from $[1,n]$ to $[n+1,2n]$ was proved in \cite{DB63} (see also \cites{C71,LLP17}).  More recently, coprime mappings have been shown to be related to the lonely runner conjecture \cites{BP22,Pom22}. 

The same paper \cite{pom1} that introduced the coprime permutations suggested other interesting arithmetic constraints that could be imposed on permutations.  One example is the set $S_{\rm div}(n)$ of permutations $\pi$ of $n$ where, for each index $i$, either $i\mid \pi(i)$ or $\pi(i)\mid i$; another is the set  $S_{\rm lcm}(n)$ of permutations where $\lcm(i,\pi(i)) \leq n$ for all $i$.  (Notice that $S_{\rm div}(n) \subseteq S_{\lcm}(n)$.)  A table with exact values of the counts of $\# S_{\rm div}(n)$ and  $\# S_{\rm lcm}(n)$ up to $n=50$ is provided in Table \ref{tab:exactvals} in the Appendix, along with the corresponding values of the $n$-th roots of these counts.  

Pomerance revisits these two sets of permutations 
in \cite{pom2}.  In each case he shows that the count of these permutations grows geometrically. In particular, for large $n$, $$ 1.9364^n \leq \# S_{\rm div}(n) < \# S_{\rm lcm}(n)  \leq 13.5597^n. $$ A heuristic argument is given that $\# S_{\rm lcm}(n) \leq 8.3830^n$, though data suggests even this bound is substantially greater than the truth.  He also gives the bounds $\# S_{\rm lcm}(n) \geq 2.1335^n$  and shows that $\# S_{\rm lcm}(n)/\# S_{\rm div}(n)> 1.00057^n$, both valid for large $n$.

In this paper, we show there exist constants $c_d$ and $c_l$ with $\#S_{\rm div}(n) = c_d^{n+o(n)}$ and $ \#S_{\rm lcm}(n) = c_l^{n+o(n)}$ as $n$ goes to infinity.  The method immediately gives the improved upper bounds $c_d < 3.31369$ and $c_l  \leq 6.60740$ without significant additional computation, but also gives an algorithm that can compute the constants to arbitrary precision, which we use to improve these upper bounds further in Section \ref{sec:numerics}.   Further ideas, discussed in Section \ref{sec:lowbds}, combined with the methods of \cite{pom2} allow us to similarly improve the lower bounds, resulting in the bounds
\[ 2.06912 < c_d < 2.69390 \ \text{ and } \ 2.30136 < c_l < 3.36352.\]

The key idea is a graph-theoretic result, bounding the ratio by which the count of the vertex-disjoint directed cycle covers of a graph can increase when a vertex is added, depending on the degree of that vertex.  This result may also be of independent interest.  A related question is posed in Section \ref{sec:question}.

\section{Notation and statement of results}

It will be convenient to consider the permutations $\pi$ in $S_{\rm div}(n)$ as vertex-disjoint directed cycle covers of the divisor graph $\mathcal{D}_{[1,n]}$ consisting of vertices $v_1, v_2, \ldots v_n$, and an edge between $v_i$ and $v_j$ if $i\mid j$ or $j \mid i$.  Every vertex in the graph is treated as having a self-loop.  In the context of this graph, permutations in $S_{\rm div}(n)$ are in bijection with vertex-disjoint directed cycle covers of $\mathcal{D}_{[1,n]}$. (The permutation $\pi \in S_{\rm div}(n)$ is associated with the cycle cover containing all directed edges of the form $i \to \pi(i)$.)

To account for fixed points and transpositions in the permutation, we allow cycles of any positive length in our cycle covers. In particular, cycles of length 1 (a single vertex involving only a loop of $\mathcal{D}_{[1,n]}$) and 2 (two vertices connected by an edge, in which case that edge is used two times) are allowed.  In more generality, we write $\mathcal{D}_{[k,n]}$ for the divisor graph consisting of integers from the interval $[k,n]$, and $S_{\rm div}([k,n])$ for the permutations restricted to that interval.

In the same way, we let  $\mathcal{L}_{[k,n]}$ be the graph consisting of integers from the interval $[k,n]$ with an edge between vertices $v_i$ and $v_j$ if $\lcm(i,j)\leq n$, and $S_{\rm lcm}([k,n])$ is the set of permutations subject to the same restriction.  Then again we can view permutations in $S_{\rm lcm}(n)$ as vertex-disjoint directed cycle covers of $\mathcal{L}_{[1,n]}$.

For an arbitrary graph $G$, let $C(G)$ denote the number of vertex-disjoint directed cycle covers of the graph $G$.  We will be interested in understanding how the count $G(G)$ grows as new vertices are added to the graph.  For a graph $G$ and a vertex $v$ of $G$ we define 
\[R(G,v) \vcentcolon = \frac{C(G)}{C(G-\{v\})}\]
to be the ratio of the number of vertex-disjoint directed cycle covers of $G$, divided by the number of such covers in the graph where the vertex $v$ has been removed (or alternatively cycle covers of $G$ in which we insist that the vertex $v$ be part of a cycle of length 1).

Intuitively, one would expect the ratio $R(G,v)$ to be larger for well connected vertices with high degree and smaller when the vertex $v$ has small degree.  It isn't obvious, however, that the ratio $R(G,v)$ would be bounded for all graphs $G$ and vertices $v$ having a fixed degree $d=d(v)$.  This somewhat surprising result is given by the following theorem.

\begin{theorem} \label{thm:graphbd}
Let $G$ be a graph with a self-loop on every vertex and $v$ a vertex of $G$ of degree $d=d(v)$, counting the self-loop once.  Then 
$$R(G,v) \leq 1 + \frac{d^2-d}{2}.$$
\end{theorem}

This theorem will be the main ingredient  used to obtain a more precise estimate for $\# S_{\rm div}(n)$ and $\# S_{\rm lcm}(n)$.

\begin{theorem} \!\label{thm:divlimit}
There exists an effectively computable constant $c_d = \!\displaystyle{\lim_{n\to \infty}} \!\left(\# S_{\rm div}(n)\right)^{1/n}$ and, for $\epsilon >0$, we have \[\#S_{\rm div}(n) = c_d^{n\left(1 + O\left(\exp((-1+\epsilon)\sqrt{\log n \log \log n}\right)\right)}.\]
\end{theorem}

We obtain an identical result for $\# S_{\rm lcm}(n)$.

\begin{theorem} \!\label{thm:lcmlimit}
There exists an effectively computable constant $c_l = \!\displaystyle{\lim_{n\to \infty}} \!\left(\# S_{\rm lcm}(n)\right)^{1/n}$ and, for $\epsilon >0$, we have \[\#S_{\rm lcm}(n) = c_l^{n\left(1 + O\left(\exp((-1+\epsilon)\sqrt{\log n \log \log n}\right)\right)}.\]
\end{theorem}

\section{Improved upper bounds} \label{sec:upbds}

Before proving these results, we can get a preview of how Theorem \ref{thm:graphbd} will be used by applying it to get improved upper bounds for $\# S_{\rm div}(n)$ and $\# S_{\rm lcm}(n)$. We start by writing $\# S_{\rm div}(n)$ as a telescoping product.  This gives\footnote{Since the interval $[n+1,n]$ is the empty set, we set $\#S_{\rm div}([n+1,n]) = C(\mathcal{D}_{[n+1,n]})=1$.}
\begin{align*}
    \# S_{\rm div}(n) = \prod_{a=1}^n \frac{\# S_{\rm div}([a,n])}{\# S_{\rm div}([a+1,n])} &= \prod_{a=1}^n \frac{C(\mathcal{D}_{[a,n]})}{C(\mathcal{D}_{[a+1,n]}))} \\
    &=\prod_{a=1}^n R(\mathcal{D}_{[a,n]},v_a) \leq \prod_{a=1}^n\left(1+\frac{d(v_a)^2-d(v_a)}{2}\right)
    \end{align*}
    using the bound from Theorem \ref{thm:graphbd}.  Noting that the vertex $v_a$ is connected to each of the $\left\lfloor \frac{n}{a} \right\rfloor$ multiples of $a$ less than $n$ (including itself) in $\mathcal{D}_{[a,n]}$ and subsequently grouping together vertices with the same degree we get 
    \begin{align*}
     \# S_{\rm div}(n) &\leq \prod_{a=1}^n\left(1+\frac{d(v_a)^2-d(v_a)}{2}\right) \ = \ \prod_{a=1}^n\left(1+\frac{\left\lfloor \frac{n}{a}\right\rfloor^2-\left\lfloor \frac{n}{a}\right\rfloor}{2}\right) \\
    &= \exp\left(\sum_{a=1}^n \log\left( 1+\frac{\left\lfloor \frac{n}{a}\right\rfloor^2-\left\lfloor \frac{n}{a}\right\rfloor}{2}\right) \right)\\
    &= \exp\left(\sum_{k=1}^{n} \left(\left\lfloor\frac{n}{k\vphantom{+1}}\right\rfloor - \left\lfloor\frac{n}{k{+}1}\right\rfloor\right)\log\left( 1+\frac{k^2 - k}{2}\right) \right). 
    \end{align*}
    Finally, taking $n$-th roots and letting $n\to \infty$ we get that 
    \begin{align*}
    \lim_{n\to \infty}\left(\# S_{\rm div}(n) \right)^{1/n} &\leq \lim_{n \to \infty} \exp\left(\frac{1}{n} \sum_{k=1}^{n} \left(\left\lfloor\frac{n}{k\vphantom{+1}}\right\rfloor - \left\lfloor\frac{n}{k+1}\right\rfloor\right)\log\left( 1+\frac{k^2 - k}{2}\right) \right) \\
    &= \lim_{n \to \infty}\exp\left( \sum_{k=1}^{n} \left(\frac{\log\left( 1+\frac{k^2 - k}{2}\right)}{k(k+1)}\right) + O\left(\frac{\log n}{\sqrt{n}}\right) \right)\\
    &= \exp\left(\sum_{k=1}^{\infty} \frac{\log\left( 1+\frac{k^2 - k}{2}\right)}{k(k+1)} \right) < \exp(1.19806) < 3.31369.
\end{align*}
Note that when we remove the floor signs in the second line above, there are $\sqrt{n}$ values of $k \leq \sqrt{n}$, which each contribute $O\left(\log n\right)$ to the sum, and the difference $\left(\left\lfloor\frac{n}{k\vphantom{+1}}\right\rfloor - \left\lfloor\frac{n}{k+1}\right\rfloor\right)$ is 0 for all but at most $\sqrt{n}$ values of $k>\sqrt{n}$, which again each contribute $O(\log n)$. This is how the error term in that line is obtained.  

The numerical bound was obtained by evaluating the sum exactly for $1\leq k \leq 10^7$ and then numerically bounding the contribution from the tail $k>10^7$.  We obtain further improved bounds below in Section \ref{sec:numerics}.

Computing improved upper bounds for $\# S_{\rm lcm}(n)$ follows similarly, however there is some additional complexity, as the degree of $v_a$ in $\mathcal{L}_{[a,n]}$ is not so simply described as in $\mathcal{D}_{[a,n]}$.  

\begin{prop} \label{prop:lcmbound}
The degree of $v_a$ in the graph $\mathcal{L}_{[a,n]}$ is 
\[d(v_a) = \sum_{1\leq j \leq \frac{n}{a}} \sum_{\substack{1 \leq \ell \leq j\\ \gcd(j,\ell) =1 \\ \ell \mid a} } 1. \]
\end{prop}
\begin{proof}
The vertex $v_a$ is connected to $v_b$ in $\mathcal{L}_{[a,n]}$ if $a\leq b \leq n$ and there are integers $j,\ell$ with $\lcm(a,b) = ja = \ell b \leq n$. In this case we must have $j \leq \frac{n}{a}$, $\ell \leq j$ (since $a\leq b$), $\gcd(j,\ell)=1$, and so $\ell \mid a$.  Having chosen values of $a$ and $j$ satisfying these conditions, $b=\frac{ja}{\ell}$ is uniquely determined and has $\lcm(a,b)\leq n$. Thus summing over all possible pairs $\ell$ and $j$ gives the result.
\end{proof}

Since \[\sum_{\substack{1 \leq \ell \leq j\\ \gcd(j,\ell) =1 \\ \ell \mid a} } 1 \leq \sum_{\substack{1 \leq \ell\leq j\\ \gcd(j,\ell) =1 } } 1 = \varphi(j),\] 
we set $\Phi_k = \sum_{j=1}^k \varphi(j)$, and the proposition implies that the degree of $v_a$ in $\mathcal{L}_{[a,n]}$ satisfies $d(v_a) \leq \Phi_{\left\lfloor\frac{n}{a}\right\rfloor}$.  Following the same logic as for $\# S_{\rm div}(n)$, we get 
\begin{align*}
     \# S_{\rm lcm}(n) &\leq  \prod_{a=1}^n\left(1+\frac{\Phi_{\left\lfloor \frac{n}{a}\right\rfloor}^2-\Phi_{\left\lfloor \frac{n}{a}\right\rfloor}}{2}\right) \\
    &= \exp\left(\sum_{k=1}^{n} \left(\left\lfloor\frac{n}{k\vphantom{+1}}\right\rfloor - \left\lfloor\frac{n}{k{+}1}\right\rfloor\right)\log\left( 1+\frac{\Phi_k^2 - \Phi_k}{2}\right) \right). 
    \end{align*}
    So, taking $n$-th roots and letting $n\to \infty$ gives
    \begin{align}
    \lim_{n\to \infty}\left(\# S_{\rm lcm}(n) \right)^{1/n} &\leq  \exp\left(\sum_{k=1}^{\infty} \frac{\log\left( 1+\frac{\Phi_k^2 - \Phi_k}{2}\right)}{k(k+1)} \right)  \label{eq:firstlcmbound} \\[0.5em]
    &< \exp(1.88819) < 6.60740. \nonumber
\end{align}
Again, the final bound was obtained by evaluating the sum exactly for $1\leq k \leq 10^7$ and numerically bounding the contribution from the tail $k>10^7$. We improve this bound further by handling the degrees of the vertices more carefully (and utilizing further numerical computations) in Section \ref{sec:numerics}.

\section{Graph Theory}
In this section, we prove Theorem \ref{thm:graphbd}, the bound for the ratio $R(G,v)$ of the number of vertex-disjoint directed cycle covers of $G$ including $v$ to those with $v$ removed (or fixed).  The result is true for graphs in general, not just those obtained as divisor or lcm-graphs.

Suppose that $v$ is a vertex of $G$ of degree $d$ (counting the loop on $v$.)  Our goal is to show that \[R(G,v) = \frac{C(G)}{C(G\setminus \{v\})} \leq 1+\frac{d^2-d}{2}.\]

Let $W=\{w_1,w_2,\ldots, w_{d-1}\}$ be the set of vertices adjacent to $v$ in $G$, not including $v$ itself.  Write $C(G)$, the total number of vertex-disjoint directed cycle covers of $G$, as \[C(G) = C_v + \sum_{i=1}^{d-1} C_{vw_i} + \sum_{\substack{1\leq i,j <d\\ i \neq j}} C_{w_ivw_j}\] where $C_v = C(G \setminus \{v\})$ is the number of vertex-disjoint directed cycle covers of $G$ in which $v$ is part of a cycle of length 1, $C_{vw_i}$ is the count of those cycle covers in which $v$ is part of a cycle of length 2 along with its neighbor $w_i$, and $C_{w_ivw_j}$ counts those where $v$ is part of a cycle of length greater than 2 and contains the edges $w_i \to v \to w_j$.  

The cycles counted by $C_{vw_i}$ are in bijection with the of vertex-disjoint directed cycle covers of $G$ in which both $v$ and $w_i$ are fixed (each part of cycles of length 1). Furthermore, the number of such cycles is less than $C_v$, the number of cycle covers where only $v$ is required to be fixed, so $C_{vw_i} \leq C_v$.  

In Lemma \ref{lem:colors} below we consider $C_{w_ivw_j}$ and show, using an argument suggested by Petrov on Mathoverflow \cite{Petrov}, that $\left(C_{w_ivw_j}\right)^2 \leq \frac{1}{4}\left(C_v\right)^2$, and thus $C_{w_ivw_j} \leq \frac{C_v}{2}$.

Inserting these inequalities $C_{vw_i} \leq C_v$ and $C_{w_ivw_j} \leq \frac{C_v}{2}$ into the definition of $R(G,v)$, we find that 
\begin{align*}
    R(G,v) &= \frac{1}{C_v}\left(C_v + \displaystyle{\sum_{i=1}^{d-1} C_{vw_i} + \sum_{\substack{1\leq i,j <d\\ i \neq j}}} C_{w_ivw_j}\right) \\
    &\leq 1 + (d-1) + (d-1)(d-2)\frac{1}{2} = 1+\frac{d^2 -d}{2}
\end{align*}
as desired. 

It remains to demonstrate the bound mentioned above, comparing $C_{w_ivw_j}$ to $C_v$.

\begin{lemma} \label{lem:colors}
Fix distinct neighbors $w_i$ and $w_j$ of $v$ in $G$ and let $C_v$ and $C_{w_ivw_j}$ be defined as above.  Then \[\left(C_{w_ivw_j}\right)^2 \leq \frac{1}{4}\left(C_v\right)^2.\]
\end{lemma}
\begin{proof}
We first break up the count $C_v$, depending on whether or not each of the vertices $w_i$ and $w_j$ are fixed (a cycle of length 1) or part of a larger cycle. We write 
\[C_v = X + Y_i + Y_j + Z\]
where $X$ counts the contribution from the cases where $v$, $w_i$, and $w_j$ are all part of 1-cycles, $Y_i$ (respectively $Y_j$) counts those where only $v$ and $w_i$ ($w_j$) are fixed, and $Z$ counts those where $v$ is fixed but  neither $w_i$ nor $w_j$ is fixed.  We can then write \begin{align*}
    (C_v)^2 = (X + Y_i + Y_j + Z)^2 &\geq (X+Z)^2 + (Y_i+Y_j)^2 \geq 4(XZ +Y_iY_j)
\end{align*}
by the AM-GM inequality.  So it remains to show that $XZ +Y_iY_j \geq (C_{w_ivw_j})^2$.  We prove this by creating an injection from pairs of cycle-covers counted by the latter to those counted by the former. 

Each of the products $XZ$, $Y_iY_j$ and $(C_{w_ivw_j})^2$ count pairs of vertex-disjoint directed cycle covers. We color the one contributed by the first entry in the product blue, and the one from the second red.  Drawing them together on the same graph results in a colored, directed multigraph in which every vertex has both an inward pointing blue and red edge, and a blue and a red edge directed away from it. Note that edges of $G$ may be used multiple times. When this construction is done for a pair of coverings counted by $(C_{w_ivw_j})^2$, the result will contain both red and blue edges $w_i\to v \to w_j$.

We can now take this collection of red-and-blue-colored, directed edges and, instead of viewing it as two monochromatic vertex-disjoint directed cycle covers, we partition it into alternating-color directed cycles.  There is a unique way to do this: at each vertex there is precisely one inward directed edge of each color and one outward directed edge of each color, so the cycles are found by alternately following blue and red edges.  Every vertex is visited exactly twice by the resulting collection of alternating-color directed cycles. (A vertex may appear twice on the same cycle, or on two different cycles.)

The injection from objects counted by $(C_{w_ivw_j})^2$ to those counted by $XZ +Y_iY_j$ now proceeds as follows:  Remove all four edges incident either to or from $v$. Then add two loops to $v$, one of each color.  The result is a graph in which every vertex except for $w_i$ and $w_j$ has an inward and outward oriented edge of each color, while $w_i$ has two inward pointing edges of each color, and $w_j$ has two outward pointing edges of each color.  This graph can thus be decomposed into alternating-color directed cycles, as well as two alternating-color directed paths from $w_j$ to $w_i$.  Let $P$ be the path that ends in a blue edge at $w_i$.  

Now add a loop to each of the vertices $w_i$ and $w_j$.  While we can't initially assign these loops colors consistent with the coloring of the remainder of the edges (as each vertex is adjacent to precisely one blue and one red edge) we color the new loop on $w_i$ blue.  We then recolor every edge in the path $P$ and reverse their directions, so that now the coloring and orientation of edges at $w_i$ (and every intermediate vertex along $P$) is consistent.  Since we have flipped the color and direction of one of the two edges adjacent to $w_j$, both edges now have the same color and opposite orientations.  We can thus color the new loop at $w_j$ the opposite color, resulting in a consistent coloring at that vertex as well. 

The result is a new collection of alternating-color directed cycles of $G$, with $v$ having two self-loops, and $w_i$ and $w_j$ each having a single loop.  We can now separate this collection of edges back into two monochromatic vertex-disjoint directed cycle covers.  We end up with two cases, based on whether the loops added to $w_i$ and $w_j$ are the same color or not.  If they are both blue, then the red coloring fixes only $v$ and neither $w_i$ nor $w_j$, so this pair is counted by $XZ$.  On the other hand, if the loops have opposite colors, then the resulting coloring is counted by $Y_iY_j$.

Every step of this construction is reversible, so it is easily checked that this is an inversion, completing the proof.  (It is not a surjection---if the collection of alternating-color directed cycles produced from a pair of colorings counted by $XZ$ or $Y_iY_j$ does not result in both of the loops on $w_i$ and $w_j$ being part of the same alternating-color directed cycle, then that pair is not in the image of this map.)
\end{proof}

\section{Asymptotic results for $\# S_{\rm div}(n)$ and $\# S_{\rm lcm}(n)$}
In \cite{mcnew}, while investigating several counting problems related to the divisibility of integers in the interval $[1,n]$, the author proved a general result about functions $f(b,m)$ which \textit{depend only on the connected component of $v_b$ in the divisor graph of the interval $[b,n]$}.  We say that a function $f:\mathbb{N}^2 \to \mathbb{R}$ has this property if $f(b,m)=f(b',m')$ whenever the connected component of $v_b$ in $\mathcal{D}_{[b,m]}$ is isomorphic to the connected component of $v_{b'}$ in $\mathcal{D}_{[b',m']}$, with an isomorphism that maps $v_b$ to $v_{b'}$.  For functions with that property, \cite{mcnew} proves the following theorem:

\begin{theorem}[Main Theorem of \cite{mcnew}] \label{thm:main} Suppose $\epsilon>0$, $A\geq0$ and $f(a,n)$ is a bounded function $|f(a,n)|\leq A$ that depends only on the connected component of $v_a$ in $\mathcal{D}_{[a,n]}$.  Then there exists a constant \begin{equation} C_f = \sum_{i=1}^\infty  \sum_{\substack{1\leq d\\ P^+(d)\leq i}} \sum_{t \in [id,(i+1)d)} \left(\frac{f(d,t)}{t(t+1)}\prod_{p\leq i} \frac{p-1}{p}\right)\label{constant} \vspace{-2mm}\end{equation}
such that \begin{equation} \sum_{a=1}^n f(a,n) = nC_f + O_\epsilon\left(An\exp\left(-(1-\epsilon)\sqrt{\log n \log \log n}\right)\right).\label{asymptotic} \end{equation}
\end{theorem}

Here $P^{+}(d)$ denotes the largest prime divisor of $d$.  The same result applies when divisor graphs are replaced by lcm graphs.

Letting $f(a,n)=\log\left(\frac{\#S_{\rm div}([a,n])}{\#S_{\rm div}([a+1,n])}\right)$, it is easy to see that the function $f$ depends only on the connected component of $v_a$ in the divisor graph.  However, it does not satisfy the requirement that $|f(k,n)|\leq A$ be bounded.  For example, even in the case of $f(1,k)$, adding $v_1$ to the divisibility graph $\mathcal{D}_{[2,k]}$ connects all of the previously unconnected vertices $v_p$ for prime integers $k/2 <p \leq k$.  Since they were previously unconnected, they were fixed by every permutation in $S_{\rm div}([2,k])$.  In $S_{\text{div}}([1,k])$  the integer 1 can be permuted with any of these primes, so we find that $f(1,k) > \log(\pi(k)-\pi(k/2))$, which goes to infinity with $k$. (Here $\pi(k)$ is the prime counting function.)

On the other hand, we can use Theorem \ref{thm:graphbd} to show that even though $f(a,n)$ is not bounded, it cannot go to infinity very quickly as a function of the ratio $\frac{n}{a}$.  In particular, we have \[f(a,n)\leq \log\left(1+\frac{\left(\frac{n}{a}\right)^2 -\frac{n}a}{2}\right)=O\left(\log \frac{n}{a}\right).\]  

In the expression for $C_f$ and throughout the proof of Theorem \ref{thm:main} in \cite{mcnew}, the variable $i$ corresponds to this ratio $i=\left\lfloor\frac an\right\rfloor$.  If we inspect that proof carefully, replacing each instance that the bound $f(a,n)\leq A$ is used with the bound $f(a,n) = O\left(\log i \right)$ instead,  we find that the result is unchanged. (In the course of the proof the $A$ that appears in each of the error terms ends up being replaced by a factor of O($\log N$), where $N$ is a parameter in the proof describing where the sum over $i$ is truncated.  This factor of $\log N$ does not affect the optimization of the error terms, it remains optimal to take $\log N = \sqrt{\log n \log \log n}$. This additional factor can be absorbed into the existing final error term.)  As a result, we can update the theorem to the form stated in Theorem \ref{thm:main2} below.

Furthermore, it is straightforward to check that the proof in \cite{mcnew} applies equally well to the graphs $\mathcal{L}_{[a,n]}$.  The key observation is that for any integers $a$ and $n$, the set of vertices in the connected component of $v_a$ in $\mathcal{D}_{[a,n]}$ is identical to the set of vertices in the connected component of $v_a$ in $\mathcal{L}_{[a,n]}$, since if two vertices $v_i$ and $v_j$ are connected by an edge in $\mathcal{L}_{[a,n]}$, they are both connected to the vertex $v_{\lcm(i,j)}$ in $\mathcal{D}_{[a,n]}$.  Combining these observations we obtain the following updated theorem.

\begin{theorem}
\label{thm:main2} Suppose $\epsilon>0$, $f(a,n)$ depends only on the connected component of $v_a$ in $\mathcal{D}_{[a,n]}$ (respectively $\mathcal{L}_{[a,n]}$) and $f(a,n) = O\left(\log \frac{n}{a}\right)$.  Then, there exists a constant given by \eqref{constant} such that the approximation \eqref{asymptotic} holds.
\end{theorem}

Theorems \ref{thm:divlimit} and \ref{thm:lcmlimit} now follow as corollaries. We take \[f(a,n)=\log\left(\frac{\#S_{\rm div}([a,n])}{\#S_{\rm div}([a+1,n])}\right)=\log\left(R(\mathcal{D}_{[a,n]},v_a)\right)\] (replacing $\mathcal{D}_{[a,n]}$  with $\mathcal{L}_{[a,n]}$ in the case of $\# S_{\rm lcm}(n)$),  noting that the function depends only on the connected component of $v_a$ in of the divisor graph $\mathcal{D}_{[a,n]}$ (or $\mathcal{L}_{[a,n]}$), and use Theorem \ref{thm:graphbd} (as well as Proposition \ref{prop:lcmbound} in the case of $S_{\rm lcm}(n)$) to see that $f(a,n) = O\left(\log \frac{a}{n}\right)$ so that our updated theorem applies.  Finally, as in Section \ref{sec:upbds}, we write $\#S_{\rm div}(n)$ as a telescoping product to get \[\# S_{\rm div}(n) = \prod_{a=1}^n \frac{\# S_{\rm div}([a,n])}{\# S_{\rm div}([a+1,n])} = \exp\left(\sum_{a=1}^n f(a,n)\right),\]
or a similar statement for $\# S_{\rm lcm}(n)$, and the theorems follow from the estimate above.

\section{Numerical Upper Bounds} \label{sec:numerics}

The expressions given for the constants of Theorems \ref{thm:divlimit} and \ref{thm:lcmlimit} provide a method of computing the values of the constants to any precision by computing exactly sufficiently many terms and estimating the contribution for the remainder of the series.  In practice, however, this is not such an easy task, as the values of $f(d,t)$ quickly become difficult to compute.  Computing the number of vertex-disjoint directed cycle covers $C(G)$ of a graph $G$ is equivalent to computing the permanent of the adjacency matrix of the graph, and computing matrix permanents is well known to be a computationally difficult problem

Nevertheless, computing exact values of $f(d,t)$ is tractable so long as the number of vertices in the  corresponding graphs is not too large.  

In order to improve the bound $c_d < 3.31369 $  obtained in Section \ref{sec:upbds}, we compute the exact value of $f(d,t)$ for the divisor graph for all $i,d$ with $id<50000$ and all $t \in [id,(i+1)d)$ for which the connected component of $v_d$ in  $\mathcal{D}_{[d,t]}$ had at most 50 vertices.  For all of the uncomputed values of $f(d,t)$ we use the same bound $f(d,t) \leq \log\left(1+\frac{i^2 -i}{2}\right)$, where $i=\left\lfloor \frac{t}{d} \right\rfloor$ as before.  Doing so gives the improved upper bound $c_d < 2.69390$.

\subsection{Improved bounds for $\# S_{\rm LCM}(n)$}
There is more room for improvement in the upper bound $c_l<6.60740$ obtained in \eqref{eq:firstlcmbound}.  Before resorting to numerical calculation of the exact values of $f(d,t)$, as above, we can first get a significant improvement by more carefully dealing with the precise degree of vertices rather than using the crude bound $d(v_a) \leq \Phi_{\left\lfloor \frac{n}{a} \right\rfloor}$ as we did in Section \ref{sec:upbds}. 

In particular, Proposition \ref{prop:lcmbound} tells us that the degree of $v_a$ in $\mathcal{D}_{[a,n]}$ depends on both $i=\left\lfloor \frac{n}{a}\right\rfloor$, and the set of integers $\{\ell < i : \ell \mid a\}$.  This set is completely characterized by the integer $d'=\gcd(a,M_i)$, where $M_i \coloneqq \lcm([1,i])$ is the least common multiple of the first $i$ integers.  Any integer $a$ in the interval $\left(\frac{n}{i+1},\frac{n}{i}\right]$ having $\gcd(a,M_i)=d'$ will have the same degree, which we denote $T_{i,d'}$. Namely, from Propositon \ref{prop:lcmbound},
\[T_{i,d'} \coloneqq d(v_a) = \sum_{1\leq j \leq i} \sum_{\substack{1 \leq \ell\leq j\\ \gcd(j,\ell) =1 \\ \ell \mid d'} } 1. \vspace{-3mm} \]
For a fixed $d' \mid M_i$, the proportion of integers $a$ having $d' = \gcd(a,M_i)$ is $\frac{\varphi\left(\frac{M_i}{d}\right)}{M_i}$.  Thus we can improve the upper bound \ref{eq:firstlcmbound} to 
\begin{align*}
    \lim_{n\to \infty}\left(\# S_{\rm lcm}(n) \right)^{1/n} &\leq  \exp\left(\sum_{i=1}^{\infty}  \sum_{d'\mid M_i} \left(\frac{\varphi\left(\frac{M_i}{d}\right)}{i(i{+}1)M_i}\log\left( 1+\frac{T_{i,d'}^2 - T_{i,d'}}{2}\right) \right)\right). 
\end{align*}
We first numerically compute the value of this sum for $1\leq i \leq 10000$.  Computing the inner sum over $d'\mid M_i$ quickly becomes unwieldy as $i$ gets large, however. So for $i\geq 42$ we only compute it over the restricted range $d'\leq 10000$.  For all of the uncomputed values we use the same bound $T_{i,d'}\leq \Phi_i$ as before.  We similarly use this bound for $10000<i\leq 10^7$, and numerically estimate the tail $i>10^7$, giving the bound $c_l \leq 4.25724$. 

Finally, we can numerically compute exact values of $f(d,t)$, as we did in the divisor graph case above.  Doing so for pairs $(i,d)$ with $id<50000$ and all $t \in [id,(i+1)d)$ for which the connected component of $v_d$ in  $\mathcal{L}_{[d,t]}$ had at most 44 vertices\footnote{Computing matrix permanents was slower in this case because the matrices were more dense, which necessitated a smaller maximum matrix size.} gives the improved bound $c_l \leq 3.36352 $. 

To compute matrix permanents, a parallel algorithm \cite{kaya}, optimized for sparse binary matrices, was used. The code provided by the author of that paper was modified to use exact large integer arithmetic instead of floating point arithmetic.

\section{Numerical Lower Bounds} \label{sec:lowbds}
As mentioned above, Theorems \ref{thm:divlimit} and \ref{thm:lcmlimit} give a method to compute lower bounds for $c_d$ and $c_l$ by explicitly computing partial sums of the series given for each constant, but the convergence is quite slow.  Using the same range of values used to obtain the upper bounds in the previous section (all pairs $(i,d)$ with $id<50000$ and all $t \in [id,(i+1)d)$ for which the connected component of $v_d$ in  $\mathcal{D}_{[d,t]}$ had at most 50 vertices (or in the case of $c_l$, the connected component of $\mathcal{L}_{[d,t]}$ had at most 44 vertices) gives lower bounds $1.70584 < c_d$, and $1.81576< c_l$, both substantially worse than the bounds $1.9364<c_d$ and $2.1335<c_l$ obtained by Pomerance in \cite{pom2}.

Pomerance obtains his bounds as follows. For an integer $b$, write the divisors $a_i\mid b$, as $1=a_1<a_2<\cdots a_k=b$, where $k=\tau(b)$.  Write $S_{\rm div}\left(\{a_1,a_2 \ldots a_i\}\right)$ for the permutations of the set $\{a_1,a_2 \ldots a_i\}$ satisfying the usual divisibility requirements. (Similarly for $S_{\rm lcm}(\{a_1,a_2 \ldots a_i\})$ in which case we impose the condition $\lcm(a_j,\pi(a_j))\leq a_i$, where $a_i$ is the largest element in the set.)  Then define 
\[\beta_{\rm div}(b)\coloneqq\frac{\log(\#S_{\rm div}(\{a_1,a_2, \ldots, a_k\})}{b} + \sum_{i=1}^{k-1}\left(\frac{1}{a_i} - \frac{1}{a_{i+1}}\right)\log(\#S_{\rm div}(\{a_1,a_2, \ldots, a_i\})\]
and 
\[\beta_{\rm lcm}(b)\!\coloneqq\frac{\log(\#S_{\rm lcm}(\{a_1,a_2, \ldots, a_k\})}{b} + \sum_{i=1}^{k-1}\left(\frac{1}{a_i} - \frac{1}{a_{i+1}}\right)\log(\#S_{\rm lcm}(\{a_1,a_2, \ldots, a_i\}).\]
Note that $\#S_{\rm lcm}(\{a_1,a_2, \ldots, a_k\}) = k!$, since $\lcm(a_i,a_j)\leq a_k = b$ for all $i$, $j$.  Finally, define $\alpha(p^i)$ on prime powers by 
\[\alpha(p^i) = \frac{p^{i+1}-p^i}{p^{i+1}-1} \]
and extend $\alpha$ to a multiplicative function on the positive integers. With this definition, $\alpha(b)$ gives the density of the integers $j$ satisfying the condition $v_p(j) \equiv 0 \pmod{ v_p(b)+1}$ for every prime $p$.  Pomerance then shows that for any integer $b$ (as $n\to \infty$) both \begin{equation}
    \#S_{\rm div}(n) \geq \exp\left(\beta_{\rm div}(b)\alpha(b)n+o(n)\right)\label{eq:divlbd}
\end{equation}
and
\begin{equation}\#S_{\rm lcm}(n) \geq \exp\left(\beta_{\rm lcm}(b)\alpha(b)n+o(n)\right).\label{eq:lcmlbd}
\end{equation}

\begin{table}[h!]
\footnotesize
    \centering
    \begin{tabular}{c c c c c}
    $b$     &   $\beta_{\rm div}(b)\alpha(b)$ & $e^{\beta_{\rm div}(b)\alpha(b)}$ & $\beta_{\rm lcm}(b)\alpha(b)$ & $e^{\beta_{\rm lcm}(b)\alpha(b)}$\\
    \hline
24 & 0.542689 & 1.720627 & 0.602065  &  1.825886\\
48 & 0.578121 & 1.782686 & 0.638299  &  1.893259\\
60 & 0.646855 & 1.909527 & 0.646855 & 1.909527\\
72 & 0.598295 & 1.819015 & 0.670619  &  1.955447\\
120 & 0.610358 & 1.841091 & 0.707611 & 2.029138\\
144 & 0.631752 & 1.880903 & 0.704928  &  2.023701\\
180 & 0.710735 & 2.035488 & 0.710735 & 2.035488\\
240 & 0.642829 & 1.901853 & 0.740126 & 2.096201\\
288 & 0.650370 & 1.916251 & 0.723606  &  2.061856\\
360 & 0.660646 & 1.936043 & 0.769250 & 2.158148\\
432 & 0.654328 & 1.923851 & 0.730695  &  2.076524\\
480 & 0.660864 & 1.936465 & 0.757764 & 2.133502\\
576 & 0.660597 & 1.935949 & 0.733726  &  2.082827\\
720 & 0.691601 & 1.996910 & 0.800104 & 2.225772\\
864 & 0.672306 & 1.958749 & 0.748644  &  2.114132\\
1440 & 0.708710 & 2.031369 & 0.816672 & 2.262957\\
1728 & 0.682146 & 1.978118 & 0.758293  &  2.134630\\
2160 & 0.711797 & 2.037650 & 0.822675 & 2.276582\\
3456 & 0.687472 & 1.988683 & 0.763452  &  2.145671\\
5184 & 0.690644 & 1.995001 & 0.767521  &  2.154419\\
10368 & 0.695844 & 2.005402 & 0.772521  &  2.165219\\

    \end{tabular}
    \caption{Values of $\beta_{\rm div}(b)\alpha(b)$ and $\beta_{\rm lcm}(b)\alpha(b)$ computed for 3- and 5-smooth values of $b$.  }
    \label{tab:bvals}
\end{table}

The bounds above are obtained by taking $b=480$.  We can first improve the bounds by extending the computations to larger values of $b$ as given in Table \ref{tab:bvals}.  By using the values obtained for $b=2160=2^4\times3^3\times 5$, we can improve the lower bounds to $2.037650<c_d$ and $2.276582<c_l$.

We can further improve these bounds by combining these ideas (for small primes) with exact computations for the large primes. An inspection of the proof in \cite{pom2} shows that \eqref{eq:divlbd} and \eqref{eq:lcmlbd} are in fact lower bounds for ``smooth'' versions of $\#S_{\rm div}(n)$ and $\#S_{\rm lcm}(n)$ in the following sense.  

Let $S_{\rm div}(n,s)\subseteq S_{\rm div}(n)$ be the set of permutations $\pi$ of $n$ such that for each $i$, either $i\mid \pi(i)$, or $\pi(i)\mid i$ and, furthermore, the quotient is $s$-smooth (meaning that $P^+\left(\tfrac{i}{\pi(i))}\right)\leq s$ or $P^+\left(\frac{\pi(i)}{i}\right)\leq s$, whichever is an integer).  In the same way, let $S_{\rm lcm}(n,s)\subseteq S_{\rm lcm}(n)$ be the permutations of $n$ where for each $i$, $\lcm(i,\pi(i))\leq n$ and also $P^+\left(\frac{\lcm(i,\pi(i))}{\gcd(i,\pi(i))}\right) \leq s$. Then we have
\begin{equation}
    \#S_{\rm div}(n,P^+(b)) \geq \exp\left(\beta_{\rm div}(b)\alpha(b)n+o(n)\right) \label{eq:smoothbbd}
\end{equation}
and
\begin{equation*}\#S_{\rm lcm}(n,P^+(b)) \geq \exp\left(\beta_{\rm lcm}(b)\alpha(b)n+o(n)\right).
\end{equation*}

We can then improve the lower bounds for $\#S_{\rm div}(n)$ and $\#S_{\rm lcm}(n)$ by obtaining lower bounds for each of the ratios $\frac{\#S_{\rm div}(n)}{\#S_{\rm div}(n,s)}$ and $\frac{\#S_{\rm lcm}(n)}{\#S_{\rm lcm}(n,s)}$.  
This can be done through explicit computation, similar to the methods introduced in previous sections. Fix a smoothness bound $s$ and let $\mathcal{R}_s \subset [1,n]$ be the set of $s$-rough integers up to $n$, those integers free of prime factors less than or equal to $s$.  Then, for an integer $r \in \mathcal{R}_s$ let $M_{r,n,s} = \{i \leq n: r\mid i \text{ and } P^+(\frac{i}{r})\leq s\}$  be the set of $s$-smooth multiples of $r$ up to $n$.  Note that we can then write
\begin{align*}
    \# S_{\rm div}(n,s) = \prod_{\substack{r \in \mathcal{R}_s}} \# S_{\rm div}(M_{r,n,s})
\end{align*}
since permutations in $S_{\rm div}(n,s)$ only permute integers within the sets $M_{r,n,s}$.

In a similar way, if $A_1, A_2, \ldots A_k$ are any disjoint subsets of $[1,n]$, then \[\# S_{\rm div}(n) \geq \prod_{i=1}^k \# S_{\rm div}(A_i).\]

Combining these two observations, if $D_1, D_2, \ldots, D_k$ are any disjoint sets that partition $\mathcal{R}_s$ then 
\begin{align}
    \frac{\# S_{\rm div}(n)}{\#S_{\rm div}(n,s)} = \frac{\# S_{\rm div}(n)}{\prod_{i=1}^k \prod_{r \in D_i}\# S_{\rm div}(M_{r,n,s})} \geq \prod_{i=1}^k \frac{\# S_{\rm div}\left(\bigcup_{r \in D_i}M_{r,n,s}\right)}{ \prod_{r \in D_i}\# S_{\rm div}(M_{r,n,s})}  \label{eq:smoothratbd}
\end{align}
since each of the sets $M_{r,n,s}$ are disjoint and partition the interval $[1,n]$. Furthermore, each of the terms in the rightmost product above are greater than or equal to 1 since the permutations enumerated by the denominator are a subset of those counted by the numerator.  So, multiplying over a set of disjoint $D_i$ which don't necessarily partition $\mathcal{R}_s$, we still obtain a lower bound for the ratio.

We now use the following strategy to pick sets $D_i$.  Consider $\mathcal{D}_{\mathcal{R}_s}$ the restriction of the divisor graph to the set $\mathcal{R}_s$. For an integer $a\in \mathcal{R}_s$, take $D_{a,n}$ to be the integers in the connected component of $v_a$ in $\mathcal{D}_{\mathcal{R}_s\cap [a,n]}$ (further restricting to integers greater than or equal to $a$).   Now we take $A_s$ to be the set of integers $a\in \mathcal{R}_s$ satisfying:

\begin{itemize}
    \item $|\bigcup_{r \in D_{a,n}} M_{r,n,s} | \leq 50$
    \item If $i = \left\lfloor \frac{n}{a} \right\rfloor$ and $d$ is the largest $i$-smooth divisor of $a$, then $id<50000$.
    \item $D_{a,n}$ is disjoint from any $D_{b,n}$ with $b<a$ that satisfies the first two properties.

\end{itemize}

We note that in this setup with $i = \left\lfloor \frac{n}{a} \right\rfloor$, $d$ the largest $i$-smooth divisor of $a$ and $t= \left\lfloor \frac{n d}{a}\right\rfloor$, the set  $\bigcup_{r \in D_{a,n}} M_{r,n,s}$ has the same exact multiplicative structure as $\bigcup_{r \in D_{d,t}} M_{r,t,s}$.  So, using the same techniques as previous sections, we can compute 

\begin{align}
 \lim_{n \to \infty}& \left(\prod_{a \in A_S} \frac{\# S_{\rm div}\left(\bigcup_{r \in D_a}M_{r,n,s}\right)}{ \prod_{r \in D_{d,t}}\# S_{\rm div}(M_{r,n,s})}\right)^{1/n} \nonumber \\
 &\hspace{-2mm}=\hspace{-1mm}\prod_{\substack{i,d,t  \\ 1\leq d \leq \frac{50000}{i}\\ d \in \mathcal{R}_s, \ P^+(d)\leq i\\ t \in [id,(i+1)d)\\ \left|\parbox{1cm}{$ \bigcup\limits_{r \in D_{d,t}}$} M_{r,t,s} \right| \leq 50}} \hspace{-3mm}\left(\!\frac{\# S_{\rm div}\!\left(\parbox{1.85cm}{$\bigcup\limits_{r \in D_{d,t}}\!M_{r,t,s}$}\right)}{ \prod\limits_{r \in D_{d,t}}\!\# S_{\rm div}(M_{r,t,s})} \!\Bigg/ \frac{\# S_{\rm div}\!\left(\parbox{1.95cm}{$\bigcup\limits_{\substack{r \in D_{d,t} \\ r>d}}M_{r,t,s}$}\right)}{ \prod\limits_{\substack{r \in D_{d,t}\\r>d}}\!\# S_{\rm div}(M_{r,t,s})} \right)^{\hspace{-1mm}\frac{1}{t(t+1)}\prod\limits_{p\leq i}\frac{p-1}{p}}\hspace{-1mm}. \label{eq:smoothratioapprox}
\end{align}

Setting $s=3$ we can evaluate the expression above exactly and find that it evaluates to 0.031282\ldots. Using this value in \eqref{eq:smoothratbd}, along with the bound for $\#S_{\rm div}(n,3)$ obtained using $b=10368$ (see Table \ref{tab:bvals}) we compute that \[c_d \geq \exp(0.695844+0.031282) = 2.06912.\]
We can also evaluate the expression in \eqref{eq:smoothratioapprox} when $s=5$, in which case we obtain the value  0.00905\ldots.  However using this, along with the value of $\beta_{\rm div}(b)\alpha(b)$ for $b=2160$, the largest 5-smooth $b$-value we were able to compute, does not result in a better lower bound for $c_d$.

Following the same strategy for $\# S_{\rm lcm}(n)$, we take $L_b$ to be the integers in the connected component of $v_b$ in $\mathcal{L}_{\mathcal{R}_s\cap [b,n]}$, and $B_s$ the set of integers $b\in \mathcal{R}_s$ satisfying analogous conditions as $A_s$, but with the size restricted to 44 instead of 50 in the first condition.  We then compute 
\begin{align*}
    \lim_{n \to \infty} \left(\prod_{a \in B_3} \frac{\# S_{\rm lcm}\left(\bigcup_{d \in L_a}M_{d,n,s}\right)}{ \prod_{d \in L_a}\# S_{\rm lcm}(M_{d,n,s})}\right)^{1/n} &= 0.040298\ldots, \\
    \lim_{n \to \infty} \left(\prod_{a \in B_5} \frac{\# S_{\rm lcm}\left(\bigcup_{d \in L_a}M_{d,n,s}\right)}{ \prod_{d \in L_a}\# S_{\rm lcm}(M_{d,n,s})}\right)^{1/n} &= 0.010828\ldots.
\end{align*}
Comparing these to the values in Table \ref{tab:bvals}, we find that this time around the best lower bound is obtained using $s=5$ and $b=2160$, in which case we obtain \[c_l \geq \exp(0.822675+0.010828) > 2.30136. \]


It isn't clear whether these methods can be used to improve the lower bound $c_\ell/c_d > 1.00057$ obtained in \cite{pom2}.  

\begin{question}
Can the lower bound for $c_\ell/c_d$ be improved?
\end{question}

\section{Final remarks on Theorem \ref{thm:graphbd}} \label{sec:question}
The bound obtained in Theorem \ref{thm:graphbd} does not seem to be sharp for any $d>2$.  For example, among all graphs $G$ of order at most 10 and vertices $v \in G$ of degree 3 (counting the loop on $v$), the bound $R(G,v)\leq 3$ holds (compared to the bound $R(G,v)\leq 4$ given by the theorem).  $R(G,v)=3$ when $G=K_3$, with loops added to each vertex.  Among graphs of order at most 10 and vertices of degree 4 we have $R(G,v)\leq \frac{19}{4}$ (compared to 7), achieved when $G=K_{2,3}$ with loops added to each vertex (and $v$ being one of the two vertices of degree 4).

The statement of Theorem \ref{thm:graphbd} is reminiscent of Bregman's inequality, which states, in the notation of this paper, that \[C(G) \leq \prod_{v \in V(G)} (d(v)!)^{1/d(v)}.\]  This comparison might lead one to conjecture that $R(G,v)=O(d(v))$, however this is false: Fixing $d>k>1$ and letting $v$ be one of the $k$ vertices of degree $d+1$ in the complete bipartite graph $K_{d,k}$ (with loops added to each vertex), we find that \begin{align*}
    R(K_{d,k},v) = \frac{C(K_{d,k})}{C(K_{d,k-1})} = \frac{\sum_{i=0}^k \binom{k}{i} \frac{d!(k-i)!}{(d-k+i)!}}{\sum_{i=0}^{k-1} \binom{k-1}{i} \frac{d!(k-1-i)!}{(d-k+i+1)!}}.
\end{align*}
Taking $k=\frac{d}{2}$, we find that the expression above simplifies to $\frac{d^2}{4}+O(d)$.  So the bound in Theorem \ref{thm:graphbd} has the optimal order of growth, though it is not clear what the optimal constant in such a bound would be.

\begin{question}
What is the best possible constant $c$ in a bound of the form $ R(G,v) \leq cd^2(1+o(1))$ as $d(v) \to \infty$? 
\end{question}

From the example above and Theorem 2.1, we know that the optimal constant is between $\frac{1}{4}$ and $\frac{1}{2}$.  It would also be of interest to determine, for fixed degrees $d$, the best upper bound of this sort.  

\section*{Acknowledgments}
The author is grateful to Carl Pomerance for helpful comments and discussions related to this paper, and to the anonymous referee for a careful reading of the manuscript and helpful suggestions.  
\renewcommand{\eprint}[1]{\href{https://arxiv.org/abs/#1}{arXiv:#1}}
\bibliographystyle{amsplain}
\bibliography{refs}

\appendix
\section{Table of values of $\# S_{\rm div}(n)$ and  $\# S_{\rm lcm}(n)$}
\begin{table}[H]
\vspace{-3mm}
\scriptsize
    \centering
\begin{tabular}{r r c r c}
$n$ & $\# S_{\rm div}(n)$ & $(\# S_{\rm div}(n))^{1/n}$ &  $\# S_{\rm lcm}(n)$ & $(\# S_{\rm lcm}(n))^{1/n}$\\
\hline

1  & 1                  & 1.000000 & 1                    & 1.000000 \\
2  & 2                  & 1.414214 & 2                    & 1.414214 \\
3  & 3                  & 1.442250 & 3                    & 1.442250 \\
4  & 8                  & 1.681793 & 8                    & 1.681793 \\
5  & 10                 & 1.584893 & 10                   & 1.584893 \\
6  & 36                 & 1.817121 & 56                   & 1.955981 \\
7  & 41                 & 1.699799 & 64                   & 1.811447 \\
8  & 132                & 1.841076 & 192                  & 1.929357 \\
9  & 250                & 1.846876 & 332                  & 1.906016 \\
10 & 700                & 1.925351 & 1184                 & 2.029248 \\
11 & 750                & 1.825447 & 1264                 & 1.914155 \\
12 & 4010               & 1.996467 & 12192                & 2.190313 \\
13 & 4237               & 1.901098 & 12872                & 2.070745 \\
14 & 10680              & 1.939792 & 37568                & 2.122134 \\
15 & 24679              & 1.962555 & 100836               & 2.155631 \\
16 & 87328              & 2.036208 & 311760               & 2.204772 \\
17 & 90478              & 1.956867 & 322320               & 2.108710 \\
18 & 435812             & 2.057285 & 2338368              & 2.258544 \\
19 & 449586             & 1.983885 & 2408848              & 2.167129 \\
20 & 1939684            & 2.062465 & 14433408             & 2.280176 \\
21 & 3853278            & 2.058785 & 32058912             & 2.277331 \\
22 & 8650900            & 2.066907 & 76931008             & 2.282754 \\
23 & 8840110            & 2.004564 & 78528704             & 2.204256 \\
24 & 60035322           & 2.109115 & 919469408            & 2.363092 \\
25 & 80605209           & 2.071355 & 1158792224           & 2.304413 \\
26 & 177211024          & 2.076107 & 2689828672           & 2.305057 \\
27 & 368759752          & 2.076284 & 4675217824           & 2.281082 \\
28 & 1380348224         & 2.120451 & 21679173184          & 2.339615 \\
29 & 1401414640         & 2.067278 & 21984820864          & 2.273133 \\
30 & 8892787136         & 2.146024 & 381078324992         & 2.432393 \\
31 & 9014369784         & 2.094725 & 386159441600         & 2.364649 \\
32 & 33923638848        & 2.133429 & 1202247415040        & 2.385063 \\
33 & 59455553072        & 2.120756 & 2207841138624        & 2.366245 \\
34 & 126536289568       & 2.120976 & 4860086689536        & 2.361221 \\
35 & 207587882368       & 2.105468 & 8681783534848        & 2.342473 \\
36 & 1495526775088      & 2.178656 & 112777175188224      & 2.456628 \\
37 & 1510769105288      & 2.133868 & 113878087417856      & 2.398302 \\
38 & 3187980614208      & 2.133241 & 247857779387904      & 2.392185 \\
39 & 5415462995568      & 2.120818 & 437979951107072      & 2.373679 \\
40 & 29811240618112     & 2.171998 & 3191130554148864     & 2.441173 \\
41 & 30071845984896     & 2.131747 & 3217753817425920     & 2.389093 \\
42 & 167426899579520    & 2.181034 & 40769431338324480    & 2.485903 \\
43 & 168778036632608    & 2.142238 & 41092863780506112    & 2.434258 \\
44 & 543720217208896    & 2.162195 & 148296367650710016   & 2.456139 \\
45 & 1741288345700048   & 2.181152 & 512674391975854336   & 2.474868 \\
46 & 3618889806595872   & 2.178863 & 1089866926717622272  & 2.466703 \\
47 & 3643985571635136   & 2.143371 & 1097208951955834368  & 2.420115 \\
48 & 28167109438114448  & 2.201420 & 13688883937198881792 & 2.504231 \\
49 & 33158989380172192  & 2.173477 & 15603885890357429248 & 2.464329 \\
50 & 107833432035711440 & 2.191064 & 58892187478016638976 & 2.485428
\end{tabular}
\vspace{-2mm}

    \caption{Values of $\# S_{\rm div}(n)$ and  $\# S_{\rm lcm}(n)$ up to $n=50$.}
    \label{tab:exactvals}
\end{table}
\addresseshere

\end{document}